\documentclass[12pt,a4paper,reqno]{amsart}
\usepackage{amssymb,amsfonts,mathrsfs}
\usepackage[T1]{fontenc}
\usepackage[utf8]{inputenc}
\usepackage[english]{babel}
\usepackage{enumerate}
\usepackage{color}

\advance\textwidth30mm \advance\hoffset-14mm
\advance\textheight30mm \advance\voffset-18mm
\newtheorem{thm}{Theorem}
\newtheorem{lem}{Lemma}
\newtheorem{cor}{Corollary}
\newtheorem{prop}{Proposition}
\newtheorem{conj}{Conjecture}
\theoremstyle{remark}
\newtheorem{rem}{Remark}
\theoremstyle{definition}

\newcommand\sign{\operatorname{sign}}
\newcommand\supp{\operatorname{supp}}
\newcommand\Cdot{{\mskip2mu\cdot\mskip2mu}}
\newcommand\PW{\mathit{PW}}

\newcommand\Ld{\mathcal L^*}
\newcommand\Ai{\operatorname{Ai}}

\renewcommand\Re{\operatorname{Re}}
\renewcommand\Im{\operatorname{Im}}

\begin{document}

\title{Exact Asymptotics of the Multidimensional Nikolskii Constant}
\author{D.\,V.~Gorbachev}
\address{Saint Petersburg State University}
\email{dvgmail@mail.ru}
\date{}

\begin{abstract}
We prove the exact asymptotics of the multidimensional normalized $L^1$ Nikolskii
constant
\[
\Ld(d)=\Bigl(\frac{\pi}{2}+o(1)\Bigr)2^{-d},
\quad d\to\infty.
\]
In addition, for each fixed dimension we obtain asymptotics for the positive zeros
of the extremal function $\varphi_d$, and determine their limiting distribution
as $d\to\infty$.
The lower bound follows from the construction of an admissible function and
its asymptotic analysis. For the upper bound, we represent
$x^{d+1}\varphi_d(x)$ as the product of two solutions of the equation
$u''+V_d(x)u=0$, compare this equation with a Bessel model, and analyze
relative canonical products.
\end{abstract}

\keywords{Nikolskii constant, extremal function, Paley--Wiener space, Bessel
functions, distribution of zeros, asymptotics with respect to dimension}

\subjclass[2020]{Primary 41A17, 41A44; Secondary 30D15, 34A30}

\maketitle

\section{Introduction}

Let $\PW_r^p(\mathbb R^d)$ be the Paley--Wiener class of functions of spherical
exponential type at most~$r$ that belong to $L^p(\mathbb R^d)$. Consider the
extremal problem of finding the exact Nikolskii constant
\begin{equation}\label{eq:Cproblem}
\mathcal C_d^{-1} =\inf\bigl\{\|f\|_{L^1(\mathbb R^d)}\colon
f\in\PW_1^1(\mathbb R^d),\ f(0)=1\bigr\}.
\end{equation}

It was proved in \cite{Da21} that problem~\eqref{eq:Cproblem} has a unique
radial extremal function $\varphi_d(|\Cdot|)$. Its one-dimensional profile is
an even real entire function of exact exponential type~$1$ and has the
canonical product
\[
\varphi_d(z)= \prod_{k=1}^{\infty} \Bigl(1-\frac{z^2}{\tau_{d,k}^2}\Bigr),
\]
where $0<\tau_{d,1}<\tau_{d,2}<\cdots$ are the positive zeros of $\varphi_d$,
and $\tau_{d,k}\sim\pi k$ as $k\to\infty$ for each fixed $d$.

Let
\[
v_d=\frac{\pi^{d/2}}{\Gamma(d/2+1)},\quad \omega_{d-1}=dv_d,
\]
denote, respectively, the volume of the unit ball and the surface area of the
unit sphere in~$\mathbb{R}^{d}$. The normalized Nikolskii constant is
\begin{equation}\label{eq:Lstar}
\Ld(d)=\frac{(2\pi)^d}{v_d}\,\mathcal C_d.
\end{equation}

The known general estimates \cite{Da21} are
\begin{equation}\label{eq:known-bounds}
2^{-d}\le\Ld(d)\le
{}_1F_2\Bigl(\frac d2;\frac d2+1,\frac d2+1;
-\frac{\beta_d^2}{4}\Bigr),
\end{equation}
where $\beta_d$ is the first positive zero of $J_{d/2}$. In particular,
\[
2^{-d}\le \Ld(d)\le\bigl(\sqrt{2/e}\,\bigr)^{d(1+o(1))}.
\]

The main result of this paper is the following.

\begin{thm}\label{thm:main}
As $d\to\infty$,
\begin{equation}\label{eq:main}
2^d\Ld(d)=\frac\pi2\,(1+o(1)).
\end{equation}
\end{thm}

We note that the numerical experiments in~\cite{Go26} indicate the existence
of the next term in the asymptotic expansion.

\begin{conj}
There is an absolute constant $c_*>0$ such that
\[
2^d\Ld(d)=\frac\pi2-c_*d^{-1/3}\bigl(1+o(1)\bigr),\quad
c_*\approx0.8.
\]
\end{conj}

In the course of the proof we also obtain an independent result on the limiting
distribution of the positive zeros of the extremal function. Introduce the
discrete measure
\[
\mu_d=\frac1d\sum_{n\ge1}\delta_{\tau_{d,n}/d}.
\]

\begin{thm}\label{thm:limit-measure}
As $d\to\infty$, the measures $\mu_d$ converge weakly on compact subsets of
$(0,\infty)$ to the locally finite measure
\begin{equation}\label{eq:intro-measure}
d\mu_*(t)=
\frac1\pi\,\sqrt{(1-t^{-2})_{+}}\,dt.
\end{equation}
\end{thm}

The limiting measure already determines the exponential scale of the Nikolskii
constant.

\begin{cor}\label{cor:intro-root-asymptotic}
The following root asymptotics holds:
\[
\Ld(d)^{1/d}\to \frac12.
\]
\end{cor}

Corollary~\ref{cor:intro-root-asymptotic} is much weaker than
Theorem~\ref{thm:main}, but its proof requires only
Theorem~\ref{thm:limit-measure} and an exact identity for the norm of the
extremal function.

The lower and upper estimates in~\eqref{eq:main} are of different nature. To prove the lower bound, we construct a family of admissible radial functions
$f_{d,N}(|\Cdot|)$ of a special form and study the asymptotics of their $L^1$ norms
as $d\to\infty$.

For the upper estimate, we use results from \cite{Go26} and \cite{Gon26}, where
a differential equation for the extremal function $\varphi_d$ was derived
independently. From~\cite{Go26} we need the representation of
$x^{d+1}\varphi_d(x)$ as the product of two solutions of the equation
$u''+V_d(x)u=0$, the Wronskian formula for the corresponding solutions,
a spectral gap, an exact identity for $\varphi_d'(\tau_{d,k})$,
an inequality for $V_d$ comparing it with the Bessel model, and the
equilibrium condition for the zeros.
These results first give the
sharpened asymptotics
\[
\tau_{d,k}=\pi\Bigl(k+\frac d2\Bigr)+O_d(k^{-1}),\quad k\to\infty,
\]
and then the global comparison
\[
\tau_{d,k}\le \gamma_{d,k},\quad k\ge1,
\]
where $0<\gamma_{d,0}<\gamma_{d,1}<\gamma_{d,2}<\ldots$ is the combined
sequence of positive zeros of $J_{d/2}(x/2)\pm Y_{d/2}(x/2)$.
After this, the upper estimate reduces to evaluating the product
\[
\prod_{k=1}^{\infty}
\frac{\pi^2(k+d/2)^2}{\gamma_{d,k}^2}.
\]

\subsection*{Convention on constants}
The letters $c$, $C$, and their indexed variants denote positive constants
which, unless stated otherwise, may change from line to line in estimates. The
indices indicate possible dependence on the corresponding parameters. The
notation $A\asymp B$ means that $cB\le A\le CB$, while
$A\asymp_\lambda B$ means that the corresponding constants may depend on
$\lambda$. The notation $O_\lambda(\Cdot)$ has the analogous meaning.

\section{Lower estimate}

In this section we prove the lower estimate
\begin{equation}\label{eq:lower-second}
\liminf_{d\to\infty}2^d\Ld(d)\ge\frac{\pi}{2}.
\end{equation}

\subsection{An admissible function}
We use the binomial expansion
\begin{equation}\label{eq:binomial-series}
(1-u)^{-1/2}
=
\sum_{n=0}^{\infty} b_n u^n,
\quad
b_n=\frac{(1/2)_n}{n!}>0,
\quad |u|<1.
\end{equation}
For a fixed integer $N\ge0$ and $x\in \mathbb{R}^{d}$, define the polynomial
\[
P_N(u)=\sum_{n=0}^N b_nu^n
\]
and the radial function
\begin{equation}\label{eq:f-lower}
f_{d,N}(|x|)
=
\frac1{M}
\int_{|\xi|\le1}
e^{i\langle x,\xi\rangle}(1-|\xi|^2)^{(d+1)/2}P_N(1-|\xi|^2)\,d\xi,
\end{equation}
where $M$ is chosen so that $f_{d,N}(0)=1$. Then $f_{d,N}(|\Cdot|)$ has
spherical exponential type at most~$1$.

Let
\[
j_\nu(t)
=
\Gamma(\nu+1)\Bigl(\frac2t\Bigr)^\nu J_\nu(t), \quad j_\nu(0)=1,
\]
be the normalized Bessel function. By the formula in
\cite[App.~B.5]{Gr08}, for $\alpha>-1$ we have
\[
\int_{|\xi|\le1}
e^{i\langle x,\xi\rangle}(1-|\xi|^2)^\alpha\,d\xi
=
\pi^{d/2}\,
\frac{\Gamma(\alpha+1)}
{\Gamma(d/2+\alpha+1)}\,
j_{d/2+\alpha}(|x|).
\]
Hence, from \eqref{eq:f-lower}, if we put
\[
M_{d,n}
=
\int_{|\xi|\le1}
(1-|\xi|^2)^{(d+1)/2+n}\,d\xi
=
\pi^{d/2}\,
\frac{\Gamma((d+3)/2+n)}
{\Gamma(d+3/2+n)},
\]
we obtain the representation
\begin{equation}\label{eq:f-Bessel-series}
f_{d,N}(t)
=
\sum_{n=0}^Nc_{d,n}\,j_{d+n+1/2}(t),\quad
c_{d,n}
=
\frac{b_nM_{d,n}}{\sum_{k=0}^Nb_kM_{d,k}}.
\end{equation}
The standard Bessel asymptotics \cite[10.17.3]{DLMF} gives
\[
j_{d+n+1/2}(t)=O_{d,n}(t^{-d-n-1}), \quad t\to+\infty,
\]
so $f_{d,N}(|\Cdot|)\in\PW_1^1(\mathbb R^d)$ and hence it is admissible in
problem~\eqref{eq:Cproblem}.

For fixed $N$, uniformly for $0\le n\le N$, we have
\[
\frac{M_{d,n}}{M_{d,0}}
=
\frac{((d+3)/2)_n}{(d+3/2)_n}
=
2^{-n}\bigl(1+O_N(d^{-1})\bigr),
\]
where $(a)_n$ denotes the Pochhammer symbol \cite[5.2.4]{DLMF}. Hence, from
\eqref{eq:f-Bessel-series}, with
\[
c_n=\frac{b_n2^{-n}}{\sum_{k=0}^Nb_k2^{-k}},
\]
we get
\begin{equation}\label{eq:c-limit-simple}
c_{d,n}=c_n+O_N(d^{-1}), \quad 0\le n\le N.
\end{equation}

For what follows it is convenient to introduce the function
\begin{equation}\label{eq:HN-simple}
H_N(w)
=
\sum_{n=0}^Nc_nw^{n+1/2}
=
\frac{w^{1/2}P_N(w/2)}{P_N(1/2)},
\end{equation}
where the branch of $w^{1/2}$ is chosen continuously on the arc
\[
w=1+e^{-2i\beta},
\quad 0<\beta<\frac{\pi}{2}.
\]

\subsection{Uniform Debye--Olver estimates}
Put
\[
Q_d=d^d\Bigl(\frac2e\Bigr)^d.
\]

\begin{lem}\label{lem:uniform-Bessel-simple}
Let $A\subset[1/2,\infty)$ be a finite set. Then, as $d\to\infty$, the
following estimates hold uniformly for $a\in A$:

\begin{enumerate}[\rm (i)]
\item in the region before the turning point, $0\le t\le d$,
\begin{equation}\label{eq:Bessel-subcritical-simple}
\frac1{Q_d}
\int_0^d|j_{d+a}(t)|t^{d-1}\,dt
=O_A(d^{-1/2}).
\end{equation}

\item if $t=d\sec\beta$, $0<\beta<\pi/2$, then
\begin{equation}\label{eq:Bessel-majorant-simple}
Q_d^{-1}|j_{d+a}(d\sec\beta)| (d\sec\beta)^{d} \tan\beta \le
C_A\sqrt{\tan\beta}.
\end{equation}

\item on every compact set $K\subset(0,\pi/2)$,
\begin{equation}\label{eq:Debye-j-simple}
j_{d+a}(d\sec\beta)
=
\frac2{\sqrt{\tan\beta}}\,
\frac{(2\cos\beta)^{d+a}}{e^{d}}\,
\Bigl\{
\cos\Bigl(
d(\tan\beta-\beta)-a\beta-\frac\pi4
\Bigr)
+O_{A,K}(d^{-1})
\Bigr\}.
\end{equation}
\end{enumerate}
\end{lem}

\begin{proof}
To prove \textup{(i)} and \textup{(ii)}, we use Olver's uniform expansion for
$J_\nu(\nu z)$ as $\nu\to\infty$ in terms of the Airy function, uniform for
$z>0$ and, in particular, valid through the turning point $z=1$; see
\cite[10.20.4]{DLMF}, \cite[Sec.~11.10]{Ol97}. Together with the standard
estimates for the Airy function, this gives the following two consequences,
uniformly for $a\in A$.

(1) There are constants $C_A,c_A>0$ such that
\begin{equation}\label{eq:J-subcritical-bound}
|J_{d+a}(dz)|\le C_A d^{-1/3} \exp\,\bigl\{-c_A d(1-z)^{3/2}\bigr\}, \quad
\frac12\le z\le1.
\end{equation}
Indeed, if $s=dz/(d+a)$, then $s\le1$ and
\[
1-s=\frac{d(1-z)+a}{d+a}\ge c_A(1-z).
\]
For the variable $\zeta=\zeta(s)$ that occurs in Olver's uniform expansion, we
have $\zeta(s)^{3/2}\asymp(1-s)^{3/2}$ for $1/3\le s\le1$, and therefore
\eqref{eq:J-subcritical-bound} follows from the uniform expansion and the
estimate
\[
|\!\Ai(u)|\le C\exp(-c u^{3/2}), \quad u\ge0.
\]

(2) For $0<\beta<\pi/2$,
\begin{equation}\label{eq:J-global-majorant}
|J_{d+a}(d\sec\beta)|
\le
C_A d^{-1/2}(\tan\beta)^{-1/2}.
\end{equation}
Indeed, if $\tan\beta\le d^{-1/3}$, then
\[
|d\sec\beta-(d+a)|\le C_A d^{1/3},
\]
and the transition estimate gives
\[
|J_{d+a}(d\sec\beta)|
\le C_A d^{-1/3}
\le C_A d^{-1/2}(\tan\beta)^{-1/2}.
\]
If $\tan\beta>d^{-1/3}$, then for all sufficiently large $d$,
\[
(d\sec\beta)^2-(d+a)^2
\ge c_A d^2\tan^2\beta,
\]
and the oscillatory part of the same uniform expansion gives
\[
|J_{d+a}(d\sec\beta)|
\le
C_A\bigl((d\sec\beta)^2-(d+a)^2\bigr)^{-1/4},
\]
which again implies \eqref{eq:J-global-majorant}.

Uniformity in $a\in A$ follows from the finiteness of $A$ and the relation
$d+a\asymp_A d$.

(i) On $0\le t\le d/2$ we use $|j_{d+a}(t)|\le1$:
\[
\frac1{Q_d}\int_0^{d/2}|j_{d+a}(t)|t^{d-1}\,dt
\le
\frac1d\Bigl(\frac e4\Bigr)^d.
\]
On $d/2<t\le d$, put $t=dz$. From the definition of $j_{d+a}$ and Stirling's
formula we obtain
\begin{align*}
\frac1{Q_d}\int_{d/2}^d|j_{d+a}(t)|t^{d-1}\,dt&= 2^ae^d\Gamma(d+a+1)d^{-d-a}
\int_{1/2}^1 z^{-a-1}|J_{d+a}(dz)|\,dz\\ &\le C_A d^{1/2}
\int_{1/2}^1|J_{d+a}(dz)|\,dz.
\end{align*}
By \eqref{eq:J-subcritical-bound} and the substitution
$u=d^{2/3}(1-z)$, the last expression is at most
\[
C_A d^{1/2}d^{-1}
\int_0^\infty e^{-cu^{3/2}}\,du
=
O_A(d^{-1/2}),
\]
which proves \eqref{eq:Bessel-subcritical-simple}.

(ii) Next, by Stirling's formula, uniformly for $a\in A$,
\[
2^{d+a}\Gamma(d+a+1)(d\sec\beta)^{-d-a}
\le
C_A d^{1/2}\,
\frac{(2\cos\beta)^{d+a}}{e^d}.
\]
Indeed,
\[
\Bigl(1+\frac ad\Bigr)^{d+a}e^{-a}=O_A(1).
\]
Therefore, \eqref{eq:J-global-majorant} gives
\[
|j_{d+a}(d\sec\beta)|
\le
\frac{C_A}{\sqrt{\tan\beta}}\,
\frac{(2\cos\beta)^{d+a}}{e^d}.
\]
Multiplying by $Q_d^{-1}(d\sec\beta)^d\tan\beta$, we find
\[
Q_d^{-1}|j_{d+a}(d\sec\beta)|
(d\sec\beta)^d\tan\beta
\le
C_A\,2^a(\cos\beta)^a\sqrt{\tan\beta}
\le
C_A\sqrt{\tan\beta},
\]
which proves \eqref{eq:Bessel-majorant-simple}.

(iii) Finally, let $\beta$ belong to a compact set $K\subset(0,\pi/2)$. For all
sufficiently large $d$, define $\beta_{d,a}\in(0,\pi/2)$ by
\[
(d+a)\sec\beta_{d,a}=d\sec\beta.
\]
Then, uniformly for $a\in A$ and $\beta\in K$,
\[
\beta_{d,a}
=
\beta-\frac{a}{d}\cot\beta+O_{A,K}(d^{-2})
\]
and
\[
(d+a)(\tan\beta_{d,a}-\beta_{d,a})
=
d(\tan\beta-\beta)-a\beta+O_{A,K}(d^{-1}).
\]
Moreover, the exact equality
\[
\cos\beta_{d,a}
=
\Bigl(1+\frac ad\Bigr)\cos\beta
\]
implies
\[
\frac1{\sqrt{\tan\beta_{d,a}}}
=
\frac1{\sqrt{\tan\beta}}
\bigl(1+O_{A,K}(d^{-1})\bigr)
\]
and
\begin{align*}
\Bigl(\frac{2\cos\beta_{d,a}}e\Bigr)^{d+a}
&=
\frac{(2\cos\beta)^{d+a}}{e^d}
\exp\,\Bigl\{
(d+a)\log\Bigl(1+\frac ad\Bigr)-a
\Bigr\} \\
&=
\frac{(2\cos\beta)^{d+a}}{e^d}
\bigl(1+O_A(d^{-1})\bigr).
\end{align*}
Since for $\beta\in K$ the quantities $\beta_{d,a}$ remain in some compact set
$K'\subset (0,\pi/2)$ uniformly for $a\in A$, the Debye formula with one
remainder term \cite[10.19.6]{DLMF}, together with the preceding expansions,
gives \eqref{eq:Debye-j-simple}.
\end{proof}

The next lemma is used to average rapid oscillations.

\begin{lem}\label{lem:cos-average-simple}
Let $0<\beta_{0}<\beta_{1}<\pi/2$, $\phi(\beta)=\tan\beta-\beta$, and
\begin{equation}\label{h-props}
h\in C^1([\beta_0,\beta_1]),\quad h(\beta)\ne0\quad \text{on
$[\beta_0,\beta_1]$}.
\end{equation}
Then, as $d\to\infty$,
\begin{equation}\label{eq:cos-average-simple}
\int_{\beta_{0}}^{\beta_{1}} \bigl| \Re\,\{e^{id\phi(\beta)}h(\beta)\} \bigr|\,d\beta \to \frac2\pi
\int_{\beta_{0}}^{\beta_{1}} |h(\beta)|\,d\beta.
\end{equation}
\end{lem}

\begin{proof}
Choose a continuous branch $\theta=\arg h$ on $[\beta_0,\beta_1]$. From
\eqref{h-props} we have $\theta\in C^1([\beta_0,\beta_1])$, since
$\theta'=\Im(h'/h)$. Hence $h=|h|e^{i\theta}$.

We have
\[
\bigl|
\Re\,\{e^{id\phi(\beta)}h(\beta)\}
\bigr|
=
|h(\beta)|\,
\bigl|
\cos(d\phi(\beta)+\theta(\beta))
\bigr|.
\]
Since
\[
|\!\cos u|
=
\frac2\pi+\sum_{k\ne0}a_ke^{2iku},
\quad
a_k=O(k^{-2}),
\]
we get
\[
\begin{aligned}
\int_{\beta_{0}}^{\beta_{1}} \bigl| \Re\,\{e^{id\phi(\beta)}h(\beta)\} \bigr|\,d\beta &= \frac2\pi
\int_{\beta_{0}}^{\beta_{1}} |h(\beta)|\,d\beta +\sum_{k\ne0}a_k H_{d,k},
\end{aligned}
\]
where
\[
H_{d,k}= \int_{\beta_{0}}^{\beta_{1}} |h(\beta)|e^{2ik\theta(\beta)} e^{2ikd\phi(\beta)}\,d\beta.
\]
For each $k\ne0$, the function $|h|e^{2ik\theta}\in C^1([\beta_0,\beta_1])$,
while $\phi'(\beta)=\tan^2\beta$ is bounded away from zero on the interval in
question. Thus one integration by parts gives
\[
H_{d,k}=O_k(d^{-1}),
\quad d\to\infty.
\]
Since
\[
|H_{d,k}|
\le
\int_{\beta_{0}}^{\beta_{1}} |h(\beta)|\,d\beta,\quad
\sum_{k\ne0}|a_k|<\infty,
\]
we may pass to the limit under the sum, and $\sum_{k\ne0}a_k H_{d,k}\to
0$. This completes the proof of \eqref{eq:cos-average-simple}.
\end{proof}

\subsection{Asymptotics of the $L^{1}$ norm}
Denote
\[
\mathcal N_{d,N}
=
\|f_{d,N}(|\Cdot|)\|_1
=
dv_d\int_0^\infty |f_{d,N}(t)|t^{d-1}\,dt .
\]

\begin{lem}\label{lem:Debye-norm-simple}
For each fixed $N$, as $d\to\infty$,
\begin{equation}\label{eq:Debye-norm-simple}
\frac{\mathcal N_{d,N}}{dv_dQ_d}
\to
\frac4\pi\,\mathcal I(H_N),
\end{equation}
where
\[
\mathcal I(H_N)
=
\int_0^{\pi/2}
|H_N(1+e^{-2i\beta})|
\sqrt{\tan\beta}\,d\beta.
\]
\end{lem}

\begin{proof}
From \eqref{eq:f-Bessel-series} and Lemma~\ref{lem:uniform-Bessel-simple}\,(i),
we obtain
\begin{equation}\label{eq:below-d-simple}
\frac1{Q_d}\int_0^d
|f_{d,N}(t)|t^{d-1}\,dt\le
\sum_{n=0}^N c_{d,n}\,
\frac1{Q_d}\int_0^d
|j_{d+n+1/2}(t)|t^{d-1}\,dt= o_N(1),
\end{equation}
where we used $c_{d,n}>0$ and $\sum_{n=0}^N c_{d,n}=1$.

Let $t=d\sec\beta$. On every compact set
$K\subset(0,\pi/2)$, from \eqref{eq:Debye-j-simple},
\eqref{eq:c-limit-simple}, and the finiteness of the sum in
\eqref{eq:f-Bessel-series}, we obtain
\[
f_{d,N}(d\sec\beta)
=
\frac2{\sqrt{\tan\beta}}
\Bigl(\frac{2\cos\beta}{e}\Bigr)^d\,
\Bigl\{
\Re\bigl[
e^{i(d(\tan\beta-\beta)-\pi/4)}
H_N(1+e^{-2i\beta})
\bigr]
+o_N(1)
\Bigr\},
\]
uniformly on $K$. Here we used the identity
\[
2\cos\beta\,e^{-i\beta}=1+e^{-2i\beta}.
\]
Also,
\[
\frac{
(d\sec\beta)^{d}\tan\beta}
{Q_d}
\Bigl(\frac{2\cos\beta}{e}\Bigr)^d
=
\tan\beta.
\]

Fix $0<\delta<\pi/4$ and put
\[
h(\beta)
=
\sqrt{\tan\beta}\,
e^{-i\pi/4}H_N(1+e^{-2i\beta}).
\]
The function $h$ belongs to $C^1([\delta,\pi/2-\delta])$ and does not vanish on
this interval. Indeed, for $N=0$ this is obvious, while for $N\ge1$ the
Enestr\"om--Kakeya theorem \cite[Corollary~8.3.5]{RS02} implies that all zeros
of $P_N$ satisfy
\[
|z|\ge \min_{0\le n<N}\frac{b_n}{b_{n+1}}
=\frac{2N}{2N-1}>1,
\]
whereas $|(1+e^{-2i\beta})/2|<1$.

Applying Lemma~\ref{lem:cos-average-simple} to $h$ and using the uniformity of
the remainder $o_N(1)$ on $[\delta,\pi/2-\delta]$, we obtain
\begin{equation}\label{eq:compact-norm-simple}
\frac1{Q_d}
\int_{d\sec\delta}^{d\csc\delta}
|f_{d,N}(t)|t^{d-1}\,dt
\to
\frac4\pi
\int_\delta^{\pi/2-\delta}
|H_N(1+e^{-2i\beta})|
\sqrt{\tan\beta}\,d\beta.
\end{equation}

To pass to the integral over $(0,\pi/2)$, we use \eqref{eq:f-Bessel-series} and
Lemma~\ref{lem:uniform-Bessel-simple}\,(ii):
\begin{equation}\label{eq:global-majorant-simple}
Q_d^{-1} |f_{d,N}(d\sec\beta)| (d\sec\beta)^{d}\tan\beta \le
C_N\sqrt{\tan\beta}, \quad 0<\beta<\frac{\pi}{2}.
\end{equation}
The function $\sqrt{\tan\beta}$ is integrable on $(0,\pi/2)$. Therefore,
\eqref{eq:global-majorant-simple} shows that the contributions from
$(0,\delta)$ and $(\pi/2-\delta,\pi/2)$ tend to zero uniformly in $d$ as
$\delta\downarrow0$. Together with \eqref{eq:compact-norm-simple} and
\eqref{eq:below-d-simple}, this gives \eqref{eq:Debye-norm-simple}.
\end{proof}

\subsection{The limit as $N\to\infty$}
For every fixed $0<\beta<\pi/2$, we have
\[
\Bigl|\frac{1+e^{-2i\beta}}2\Bigr|=\cos\beta<1.
\]
It follows from the binomial series \eqref{eq:binomial-series} and
\eqref{eq:HN-simple} that, as $N\to \infty$,
\begin{equation}\label{eq:HN-limit-simple}
H_N(w)\to
H_*(w)=
\sqrt{\frac{w}{2-w}},
\quad
w=1+e^{-2i\beta}.
\end{equation}
Here we also used $P_N(1/2)\to\sqrt2$.

Moreover,
\[
|P_N(w/2)|
\le P_N(|w|/2)
\le(1-|w|/2)^{-1/2}
=(1-\cos\beta)^{-1/2}.
\]
Since $P_N(1/2)\ge1$ and $|w|=2\cos\beta$, we obtain
\begin{equation}\label{eq:HN-majorant-simple}
|H_N(1+e^{-2i\beta})|\sqrt{\tan\beta}
\le
C\frac{\sqrt{\sin\beta}}
{\sqrt{1-\cos\beta}}
\le C\beta^{-1/2}.
\end{equation}
The right-hand side is integrable on $(0,\pi/2)$.

Finally,
\[
\frac{1+e^{-2i\beta}}{1-e^{-2i\beta}}
=-i\cot\beta,
\]
and therefore
\begin{equation}\label{eq:magic-simple}
|H_*(1+e^{-2i\beta})|\sqrt{\tan\beta}=1.
\end{equation}
By \eqref{eq:HN-limit-simple}, \eqref{eq:HN-majorant-simple},
\eqref{eq:magic-simple}, and the dominated convergence theorem,
\begin{equation}\label{eq:I-HN-limit-simple}
\mathcal I(H_N)\to\frac{\pi}{2}.
\end{equation}

\subsection{Completion of the proof}
Since $f_{d,N}(|\Cdot|)$ is admissible in \eqref{eq:Cproblem},
\[
\Ld(d)
\ge
\frac{(2\pi)^d}{v_d\mathcal N_{d,N}}.
\]
By Stirling's formula, as $d\to \infty$,
\[
dv_d^2Q_d
=
\frac{(4\pi)^d}{\pi}
\bigl(1+O(d^{-1})\bigr).
\]
On the other hand, by Lemma~\ref{lem:Debye-norm-simple}, for every fixed $N$,
\[
\frac{\mathcal N_{d,N}}{dv_dQ_d}
\to
\frac4\pi\,\mathcal I(H_N).
\]
Hence
\[
\liminf_{d\to\infty}2^d\Ld(d)
\ge
\frac{\pi^2}{4\,\mathcal I(H_N)}.
\]
Letting $N\to\infty$ and using \eqref{eq:I-HN-limit-simple}, we obtain the
required lower estimate
\[
\liminf_{d\to\infty}2^d\Ld(d)
\ge
\frac{\pi}{2}.
\]

\begin{rem}
A more precise lower estimate can be obtained by replacing the factor
$P_N(1-|\xi|^2)$ in the definition of the admissible function
\eqref{eq:f-lower} by
\[
\frac{1}{\sqrt{1-(1-\lambda d^{-1/3})(1-|\xi|^2)}},
\]
where $\lambda>0$ is fixed and $d$ is sufficiently large. In this case, the
main contribution to the corresponding Bessel expansion comes from indices
$n\asymp d^{1/3}$, and a uniform analysis near the turning point gives
\[
2^d\Ld(d)\ge
\frac{\pi}{2}-\widetilde c\,d^{-1/3}\bigl(1+o(1)\bigr),
\]
where the constant $\widetilde c>0$ is determined by the choice of $\lambda$.
Numerical optimization in $\lambda$ gives $\widetilde c\approx0.8$. A proof of
this refinement requires a substantially more complicated analysis of the
transition region and is therefore not included here.
\end{rem}

\section{Upper estimate in Theorem~\ref{thm:main}}

We use the results of~\cite{Go26}. Put
\begin{equation}\label{eq:ad}
a_d=\frac1{2v_d\mathcal C_d}.
\end{equation}
It follows from this and~\eqref{eq:Lstar} that
\begin{equation}\label{eq:L-via-a}
\Ld(d)=
\frac{2^{d-1}\Gamma(d/2+1)^2}{a_d}.
\end{equation}

\begin{prop}[{\cite[Corollary~1, formula (63), Sec.~8.2]{Go26}}]
For every $d\ge1$, there exist a real function $V_d\in C((0,\infty))$ and two
real linearly independent solutions $u_{d,1}$, $u_{d,2}$ of the equation
\[
u''+V_d(x)u=0
\]
on $(0,\infty)$ such that
\begin{equation}\label{u-u-varphi}
u_{d,1}(x)u_{d,2}(x)=x^{d+1}\varphi_d(x)
\end{equation}
and the Wronskian is
\[
W(u_{d,1},u_{d,2})= u_{d,1}u_{d,2}' -u_{d,1}'u_{d,2}= 2a_d.
\]
The positive zeros of $u_{d,1}$ and $u_{d,2}$ are
$\tau_{d,1},\tau_{d,3},\ldots$ and $\tau_{d,2},\tau_{d,4},\ldots$,
respectively. Moreover,
\begin{align}
\label{eq:V-bound}
V_d(x)&\le
\frac14-\frac{d^2-1}{4x^2},\quad x>0,\\
\label{eq:V-infty}
V_d(x)&=\frac14+O_d(x^{-2}),\quad x\to +\infty.
\end{align}
\end{prop}

\begin{prop}[{\cite[Sec.~8.1]{Go26}}]
For every $n\ge1$, the identities
\begin{equation}\label{eq:residue}
\tau_{d,n}^{d+1}\varphi_d'(\tau_{d,n})
=2(-1)^n a_d
\end{equation}
and
\begin{equation}\label{eq:zero-equilibrium}
a_d
=
\tau_{d,n}^d
\prod_{k\ge1,\;k\ne n}
\Bigl|1-\frac{\tau_{d,n}^2}{\tau_{d,k}^2}\Bigr|
\end{equation}
hold.
\end{prop}

Put
\[
\sigma_d(x)=\sign\varphi_d(x)
\]
and denote the one-dimensional Fourier transform by
\[
\widehat f(t)=\int_{\mathbb R}f(x)e^{-itx}\,dx.
\]

\begin{prop}[{\cite[Lemma~1]{Go26}}]
In $\mathcal S'(\mathbb R)$, one has
\begin{equation}\label{eq:signature-gap}
\widehat{|x|^{d-1}\sigma_d}=\frac{4a_d}{d}\quad \text{on $(-1,1)$}.
\end{equation}
\end{prop}

The following important consequence follows immediately from
\eqref{eq:V-bound}, where we use standard facts about Bessel functions
\cite{Wa66}.

\begin{cor}\label{cor:Bessel-majorant}
For all $x>0$,
\begin{equation}\label{eq:Bessel-majorant}
V_d(x)\le V_{d,0}(x)=\frac14-\frac{d^2-1}{4x^2},
\end{equation}
where the potential $V_{d,0}$ corresponds to the Bessel equation in normal
form
\[
y''+\Bigl(\frac14-\frac{d^2-1}{4x^2}\Bigr)y=0
\]
on $(0,\infty)$. Linearly independent solutions of this equation are
\begin{equation}\label{eq:j-y}
y_{d,1}(x)=\sqrt{x}\,J_{d/2}(x/2),
\quad
y_{d,2}(x)=\sqrt{x}\,Y_{d/2}(x/2),
\end{equation}
and their Wronskian is
\begin{equation}\label{eq:jy-W}
W(y_{d,1},y_{d,2})=\frac2\pi.
\end{equation}
\end{cor}

\subsection{Phase of the extremal function for fixed dimension}
In this subsection $d$ is fixed, so the index $d$ is sometimes omitted from
functions.

\subsubsection*{Asymptotic integration}
We need a simple version of the standard theorem on an integrable perturbation
of the harmonic oscillator.

\begin{lem}\label{lem:asymptotic-integration}
Let $X>0$ and let a real function $u\not\equiv0$ on $(X,\infty)$ satisfy
\[
u''+\Bigl(\frac14+r(x)\Bigr)u=0,
\quad
r(x)=O(x^{-2}),\quad x\to+\infty.
\]
Then there exist $A>0$ and $\theta\in\mathbb R$ such that
\begin{align}
\label{eq:u-asymptotic}
u(x)&=A\cos\Bigl(\frac x2+\theta\Bigr)+O(x^{-1}),\\
\label{eq:up-asymptotic}
u'(x)&=-\frac A2\sin\Bigl(\frac x2+\theta\Bigr)+O(x^{-1}).
\end{align}
\end{lem}

\begin{proof}
Put
\[
Z(x)=e^{-ix/2}\bigl(u(x)-2iu'(x)\bigr).
\]
The equation directly gives
\[
Z'(x)=2ir(x)e^{-ix/2}u(x).
\]
Since $u$ is real,
\[
|u(x)|\le |u(x)-2iu'(x)|=|Z(x)|,
\]
and therefore
\begin{equation}\label{Z-r}
|Z'(x)|\le2|r(x)|\,|Z(x)|.
\end{equation}

Fix $X_0>X$. From \eqref{Z-r},
\[
|Z(x)|\le |Z(X_0)| +2\int_{X_0}^x |r(t)|\,|Z(t)|\,dt, \quad x\ge X_0,
\]
and Gronwall's inequality
\cite[Ch.~3, Th.~1.1]{Ha02} gives
\[
|Z(x)|
\le |Z(X_0)|
\exp\biggl(2\int_{X_0}^x|r(t)|\,dt\biggr).
\]
Thus $Z$ is bounded for $x\ge X_0$. Hence, for $y>x\ge X_0$,
\[
|Z(y)-Z(x)|\le C\int_x^y|r(t)|\,dt.
\]
Since $r(t)=O(t^{-2})$, the right-hand side tends to zero as
$x,y\to+\infty$. Therefore the limit
\[
Z_\infty=\lim_{x\to+\infty}Z(x)
\]
exists.

We show that $Z_\infty\ne0$. If $Z_\infty=0$, then
\[
|Z(x)|
\le 2\int_x^\infty |r(t)|\,|Z(t)|\,dt.
\]
Increasing $X_0$ if necessary, we may assume that
\[
2\int_{X_0}^\infty|r(t)|\,dt<1.
\]
Taking the supremum over $x\ge X_0$, we obtain
\[
\sup_{x\ge X_0}|Z(x)|
\le
2\biggl(\int_{X_0}^\infty|r(t)|\,dt\biggr)
\sup_{x\ge X_0}|Z(x)|,
\]
so $Z\equiv0$ on $[X_0,\infty)$. Consequently,
$u=u'=0$ on this half-line, and uniqueness implies $u\equiv0$, a contradiction.

Since
\[
\int_x^\infty|r(t)|\,dt=O(x^{-1}),
\]
we have
\[
Z(x)=Z_\infty+O(x^{-1}).
\]
Write $Z_\infty=Ae^{i\theta}$, $A>0$. Then
\[
u(x)-2iu'(x)=Ae^{i(x/2+\theta)}+O(x^{-1}),
\]
and taking real and imaginary parts gives \eqref{eq:u-asymptotic} and
\eqref{eq:up-asymptotic}.
\end{proof}

Applying Lemma~\ref{lem:asymptotic-integration} to $u_{d,1}$ and $u_{d,2}$ and
using \eqref{eq:V-infty}, we obtain numbers $A_1,A_2>0$ and
$\theta_1,\theta_2$ such that
\[
u_{d,j}(x)=A_j\cos\Bigl(\frac x2+\theta_j\Bigr)+O_d(x^{-1}), \quad j=1,2,
\]
with an analogous formula for the derivatives.

From constancy of the Wronskian,
\[
2a_d=\frac{A_1A_2}{2}\,\sin(\theta_1-\theta_2).
\]
In particular,
\[
q_d=\cos(\theta_1-\theta_2)\in(-1,1).
\]
On the other hand, by~\eqref{u-u-varphi},
\[
u_{d,1}(x)u_{d,2}(x)=x^{d+1}\varphi_d(x).
\]
Therefore
\begin{equation}\label{eq:phi-cos-q}
x^{d+1}\varphi_d(x)
=B_d\bigl(\cos(x+\Theta_d)+q_d\bigr)+O_d(x^{-1}),
\end{equation}
where
\[
B_d=\frac{A_1A_2}{2}>0,\quad \Theta_d=\theta_1+\theta_2.
\]

\subsubsection*{Spectral gap}
We need an elementary lemma on averaging a sign function.

\begin{lem}\label{lem:sign-mean}
Let $F\in C((0,\infty))$ be a real function such that, as $x\to+\infty$,
\[
F(x)=C\bigl(\cos(x+\Theta)+q\bigr)+O(x^{-1}),
\quad C\ne0,\quad |q|<1,
\]
and let $s(x)=\sign F(x)$. Then, for every $w\in L^1(0,\infty)$,
\[
\lim_{R\to\infty}\int_0^\infty w(y)s(Ry)\,dy
=\sign(C)\,m(q)\int_0^\infty w(y)\,dy,
\]
where
\begin{equation}\label{eq:mq}
m(q)=\frac1{2\pi}\int_0^{2\pi}\sign(\cos t+q)\,dt
=\frac2\pi\arcsin q.
\end{equation}
\end{lem}

\begin{proof}
Put
\[
p(x)=\sign(C)\sign(\cos(x+\Theta)+q).
\]
The zeros of $\cos(x+\Theta)+q$ are simple because $|q|<1$. Hence the
asymptotics of $F$ shows that, on every sufficiently far period, the set where
$s\ne p$ is contained in the union of two intervals of total length
$O(x^{-1})$. Therefore,
\[
\int_1^X|s(x)-p(x)|\,dx=O(\log X).
\]
For a compactly supported step function $w$, after the substitution $x=Ry$,
this gives
\[
\int_0^\infty w(y)(s(Ry)-p(Ry))\,dy\to0.
\]
Since $s$ and $p$ are bounded, the limit extends to all $w\in L^1(0,\infty)$
by density.

The function $p$ is periodic. Standard periodic averaging gives
\[
\int_0^\infty w(y)p(Ry)\,dy
\to
\biggl(\frac1{2\pi}\int_0^{2\pi}p(t)\,dt\biggr)
\int_0^\infty w(y)\,dy.
\]
Indeed, this is checked directly for compactly supported step functions $w$,
and the general case follows by density in $L^1(0,\infty)$ and boundedness of
$p$.

It remains to compute the mean value. The set
\[
\{t\in[0,2\pi)\colon \cos t>-q\}
\]
has length
\[
2\arccos(-q)=\pi+2\arcsin q,
\]
which gives~\eqref{eq:mq}.
\end{proof}

\begin{lem}
In~\eqref{eq:phi-cos-q}, one has $q_d=0$. Consequently, as $x\to +\infty$,
\begin{equation}\label{eq:phi-pure-cos}
x^{d+1}\varphi_d(x)=B_d\cos(x+\Theta_d)+O_d(x^{-1}).
\end{equation}
\end{lem}

\begin{proof}
Choose a real even function $g\in\mathcal S(\mathbb R)$,
$g\not\equiv0$, such that $\supp\widehat g\subset(-1/4,1/4)$, and put
$\eta=g^2$. Then $\eta\ge0$, $\eta$ is even, $\eta\in\mathcal S$, and
$\supp\widehat\eta\subset(-1/2,1/2)$. For $R\ge1$, put
$\eta_R(x)=\eta(x/R)$. Then
\[
\supp\widehat{\eta_R}\subset(-1/(2R),1/(2R))\subset(-1,1).
\]
From the distributional identity~\eqref{eq:signature-gap} and Parseval's
formula for distributions, we obtain the exact equality
\[
\int_{\mathbb R}|x|^{d-1}\sigma_d(x)\eta(x/R)\,dx
=\frac{4a_d}{d}\,\eta(0).
\]
After the substitution $x=Ry$ and division by $R^d$,
\begin{equation}\label{eq:weighted-sign-limit}
\int_{\mathbb R}|y|^{d-1}\sigma_d(Ry)\eta(y)\,dy
\to0.
\end{equation}
The functions $\varphi_d$ and $\sigma_d$ are even, so the left-hand side is
twice the integral over $(0,\infty)$. On the positive half-line,
$\sign\varphi_d(x)=\sign(x^{d+1}\varphi_d(x))$. Apply
Lemma~\ref{lem:sign-mean} to~\eqref{eq:phi-cos-q} with the weight
$w(y)=y^{d-1}\eta(y)$. Since $w\ge0$ and $w\not\equiv0$,
\eqref{eq:weighted-sign-limit} implies $m(q_d)=0$. By~\eqref{eq:mq},
$q_d=0$.
\end{proof}

\begin{lem}
There exists a constant $c_d\in\mathbb R$ such that
\begin{equation}\label{eq:tail-with-c}
\tau_{d,n}=\pi(n+c_d)+O_d(n^{-1}),
\quad n\to\infty.
\end{equation}
\end{lem}

\begin{proof}
The asymptotics of $u_{d,j}$ and their derivatives obtained in
Lemma~\ref{lem:asymptotic-integration}, together with
$x^{d+1}\varphi_d(x)=u_{d,1}(x)u_{d,2}(x)$, imply, in addition to
\eqref{eq:phi-pure-cos}, that
\[
\bigl(x^{d+1}\varphi_d(x)\bigr)'
=
-B_d\sin(x+\Theta_d)+O_d(x^{-1}).
\]
Therefore, near every sufficiently large zero of
$\cos(x+\Theta_d)$ there is exactly one zero of
$x^{d+1}\varphi_d(x)$, and its distance from the corresponding zero of the
cosine is $O_d(x^{-1})$. After matching the numbering, we obtain
\eqref{eq:tail-with-c}.
\end{proof}

\subsection{Exact determination of the phase shift}

The following simple lemma on canonical products will be used twice.

\begin{lem}\label{lem:derivative-exponent}
Let $x_n>0$ be strictly increasing and
\begin{equation}\label{eq:x-rho-tail}
x_n=\pi(n+c)+O(n^{-1}),\quad n\to\infty,
\end{equation}
with some $c\in\mathbb R$. Let
\[
P(z)=\prod_{n=1}^{\infty}\Bigl(1-\frac{z^2}{x_n^2}\Bigr).
\]
Then there exists $C>0$ such that
\begin{equation}\label{eq:Pprime-exponent}
|P'(x_n)|=C n^{-2c-1}(1+o(1)).
\end{equation}
\end{lem}

\begin{proof}
Choose $N$ so large that $n+c>0$ for $n\ge N$, and put
$\rho_n=\pi(n+c)$. The model tail product is
\[
Q_N(z)=\prod_{n=N}^{\infty}
\Bigl(1-\frac{z^2}{\rho_n^2}\Bigr)
=\frac{\Gamma(N+c)^2}
{\Gamma(N+c+z/\pi)\Gamma(N+c-z/\pi)}.
\]
At $\rho_n$, $n\ge N$,
\begin{equation}\label{eq:Q-tail-derivative}
|Q_N'(\rho_n)|
=\frac{\Gamma(N+c)^2\Gamma(n-N+1)}
{\pi\Gamma(n+N+2c)}
=C_N n^{1-2N-2c}(1+o(1)).
\end{equation}

By~\eqref{eq:x-rho-tail}, $x_k^2-\rho_k^2=O(1)$. For $k\ne n$,
\[
|\rho_k^2-\rho_n^2|\asymp |k-n|(k+n)
\]
for large $k,n$, and hence
\begin{equation}\label{eq:gap-sum}
\sum_{k\ge N,\;k\ne n}
\frac1{|\rho_k^2-\rho_n^2|}
=O\Bigl(\frac{\log n}{n}\Bigr).
\end{equation}
Thus, applying $\log(1+s)=s+O(s^2)$ first outside a finite set of indices, we
obtain
\begin{equation}\label{eq:gap-product-tail}
\prod_{k\ge N,\;k\ne n}
\frac{|x_k^2-x_n^2|}{|\rho_k^2-\rho_n^2|}=1+o(1).
\end{equation}
Moreover, $x_k/\rho_k=1+O(k^{-2})$, so the corresponding normalizing products
converge to a nonzero constant. Finally, the finite factor
\[
\prod_{k<N} \Bigl(1-\frac{x_n^2}{x_k^2}\Bigr)
\]
has modulus $C_0n^{2N-2}(1+o(1))$. Combining this with
\eqref{eq:Q-tail-derivative}, \eqref{eq:gap-sum}, and
\eqref{eq:gap-product-tail}, we obtain \eqref{eq:Pprime-exponent}.
\end{proof}

\begin{thm}\label{thm:true-tail}
For every fixed $d\ge1$,
\[
\tau_{d,n}=\pi\Bigl(n+\frac d2\Bigr)+O_d(n^{-1}),
\quad n\to\infty.
\]
\end{thm}

\begin{proof}
By~\eqref{eq:tail-with-c} and Lemma~\ref{lem:derivative-exponent},
\[
|\varphi_d'(\tau_{d,n})|
=C n^{-2c_d-1}(1+o(1)).
\]
On the other hand, the exact identity~\eqref{eq:residue} gives
\[
|\varphi_d'(\tau_{d,n})|
=\frac{2a_d}{\tau_{d,n}^{d+1}}
=C' n^{-d-1}(1+o(1)).
\]
Therefore $2c_d+1=d+1$, that is, $c_d=d/2$.
\end{proof}

\subsection*{Comparison with the Bessel equation}
Put $\nu=d/2$. Introduce the phase function $\theta_\nu$ by the standard
formula
\begin{equation}\label{eq:Bessel-phase-def}
J_\nu(z)=M_\nu(z)\cos\theta_\nu(z),
\quad
Y_\nu(z)=M_\nu(z)\sin\theta_\nu(z),
\end{equation}
where $M_\nu(z)>0$ and $\theta_\nu(0+)=-\pi/2$. The identity
\[
M_\nu(z)^2\theta_\nu'(z)=\frac{2}{\pi z}
\]
implies that $\theta_\nu$ is strictly increasing; see
\cite[\S\,10.18]{DLMF}.

Put
\[
P_d(x)=\theta_\nu(x/2)+\frac\pi4.
\]
Then $P_d$ is strictly increasing. For the solutions introduced in
\eqref{eq:j-y}, put
\[
y_{d}^{-}(x)=y_{d,1}(x)-y_{d,2}(x),\quad
y_{d}^{+}(x)=y_{d,1}(x)+y_{d,2}(x).
\]
From \eqref{eq:Bessel-phase-def},
\[
y_{d}^{-}(x)=A_d(x)\cos P_d(x),
\quad
y_{d}^{+}(x)=A_d(x)\sin P_d(x),
\]
where
\[
A_d(x)=\sqrt{2x}\,M_\nu(x/2)>0.
\]
Moreover, by \eqref{eq:jy-W},
\begin{equation}\label{eq:v-W}
W(y_{d}^{-},y_{d}^{+})=\frac4\pi.
\end{equation}

Define the sequence $\gamma_{d,n}$ by
\begin{equation}\label{eq:beta-def}
P_d(\gamma_{d,n})=\frac{\pi n}{2},
\quad n=0,1,2,\ldots.
\end{equation}
Then $\gamma_{d,1},\gamma_{d,3},\ldots$ are all positive zeros of
$y_{d}^{-}$, while $\gamma_{d,0},\gamma_{d,2},\gamma_{d,4},\ldots$ are all
positive zeros of $y_{d}^{+}$.

For later use, note that $\gamma_{d,1}>d$. Indeed, the classical ordering of
the first zeros $\nu\le j'_{\nu,1}<y_{\nu,1}<j_{\nu,1}$ implies
$J_\nu(\nu)>0$ and $Y_\nu(\nu)<0$ for $\nu\ge1/2$. Hence
$\theta_\nu(\nu)\in(-\pi/2,0)$ and
$P_d(d)<\pi/4<\pi/2=P_d(\gamma_{d,1})$; see
\cite[\S\,10.21]{DLMF}.

The fixed-order phase asymptotics
\[
\theta_\nu(z)
=z-\Bigl(\frac\nu2+\frac14\Bigr)\pi+O_\nu(z^{-1})
\]
gives
\begin{equation}\label{eq:beta-tail}
\gamma_{d,n}=\pi\Bigl(n+\frac d2\Bigr)+O_d(n^{-1}).
\end{equation}

\subsection{Comparison of zeros}
The main statement is as follows.

\begin{thm}\label{thm:all-zero-comparison}
For all $d\ge1$ and $n\ge1$,
\[
\tau_{d,n}\le\gamma_{d,n}.
\]
\end{thm}

Instead of applying Sturm's theorem directly, it is convenient to use its
phase form, which we now prove.

\begin{lem}\label{lem:phase-comparison}
Let $v_1,v_2$ be a fundamental system for the equation
\[
v''+V_0v=0,
\quad W(v_1,v_2)=W_0>0,
\]
and suppose
\[
v_1=A\cos P,\quad v_2=A\sin P,
\quad A>0,\quad P'>0.
\]
Assume also that, as $x\to+\infty$,
\[
v_j(x)=O(1),\quad v_j'(x)=O(1),
\quad j=1,2.
\]
Let $u$ satisfy
\[
u''+Vu=0,
\quad \Delta=V_0-V\ge0,
\]
and, as $x\to+\infty$,
\[
u(x)=c\,v_1(x)+o(1),\quad
u'(x)=c\,v_1'(x)+o(1)
\]
with some $c>0$. Then there exists a continuous real function $\delta$ such
that
\begin{equation}\label{eq:u-phase-rep}
u=rA\cos(P-\delta),\quad r>0,
\end{equation}
\begin{equation}\label{eq:delta-prime}
\delta'(x)=\frac{\Delta(x)u(x)^2}{W_0r(x)^2}\ge0,
\quad
\delta(+\infty)=0.
\end{equation}
In particular, $\delta(x)\le0$ for all $x>0$.
\end{lem}

\begin{proof}
Write $u$ by variation of constants:
\[
u=a(x)v_1(x)+b(x)v_2(x),
\quad
a'v_1+b'v_2=0.
\]
Since $u''+V_0u=\Delta u$, we obtain
\begin{equation}\label{eq:ab-prime}
a'=-\frac{\Delta u\,v_2}{W_0},
\quad
b'=\frac{\Delta u\,v_1}{W_0}.
\end{equation}
By Cramer's rule,
\[
a=\frac{uv_2'-u'v_2}{W_0},
\quad
b=\frac{v_1u'-v_1'u}{W_0}.
\]
Therefore the asymptotics of $u,u'$ and the boundedness of $v_j,v_j'$ imply
\[
a(x)=c+o(1),\quad b(x)=o(1), \quad x\to+\infty.
\]
The pair $(a,b)$ never equals $(0,0)$, since otherwise $u=u'=0$ at that point.
Thus we can choose a continuous argument
\[
a=r\cos\delta,\quad b=r\sin\delta,
\quad r>0,
\]
so that $\delta(+\infty)=0$. From~\eqref{eq:ab-prime},
\[
\delta'=\frac{ab'-ba'}{r^2}
=\frac{\Delta u(av_1+bv_2)}{W_0r^2}
=\frac{\Delta u^2}{W_0r^2}\ge0.
\]
Formula~\eqref{eq:u-phase-rep} follows from the representations of $v_1$ and
$v_2$.
\end{proof}

\begin{proof}[Proof of Theorem~\ref{thm:all-zero-comparison}]
By Corollary~\ref{cor:Bessel-majorant},
\[
\Delta_d(x)=V_{d,0}(x)-V_d(x)\ge0.
\]

First consider $u_{d,1}$. By Lemma~\ref{lem:asymptotic-integration},
\[
u_{d,1}(x)
=A_1\cos\Bigl(\frac x2+\theta_1\Bigr)+O_d(x^{-1}),
\quad
u_{d,1}'(x)
=-\frac{A_1}{2}\sin\Bigl(\frac x2+\theta_1\Bigr)+O_d(x^{-1}),
\]
where $A_1>0$. On the other hand, the standard Bessel asymptotics gives
\[
y_d^-(x)
=
2\sqrt{\frac2\pi}\,
\cos\Bigl(\frac x2-\frac{\pi d}{4}\Bigr)+O_d(x^{-1}),
\]
\[
(y_d^-)'(x)
=
-\sqrt{\frac2\pi}\,
\sin\Bigl(\frac x2-\frac{\pi d}{4}\Bigr)+O_d(x^{-1}).
\]
Theorem~\ref{thm:true-tail} and formula~\eqref{eq:beta-tail} imply that the
corresponding phase shifts agree modulo $\pi$. Hence, changing $u_{d,1}$ to
$-u_{d,1}$ if necessary, we get
\[
u_{d,1}(x)=c_1y_d^-(x)+o(1),
\quad
u_{d,1}'(x)=c_1(y_d^-)'(x)+o(1),
\quad c_1>0.
\]

The standard asymptotics of Bessel functions and their derivatives also gives
\[
y_d^\pm(x)=O_d(1),\quad (y_d^\pm)'(x)=O_d(1),
\quad x\to+\infty,
\]
so all assumptions of Lemma~\ref{lem:phase-comparison} are satisfied. Applying
it with $P=P_d$ and $W_0=4/\pi$, we obtain
\[
u_{d,1}=rA_d\cos(P_d-\delta_1),
\quad \delta_1\le0.
\]
At every zero of $u_{d,1}$, formula~\eqref{eq:delta-prime} gives
$\delta_1'=0$, and hence
\[
(P_d-\delta_1)'=P_d'>0
\]
at the zero itself. Thus every crossing of a level
$\pi/2+\pi\mathbb Z$ is in the increasing direction. The same level cannot be
crossed twice, because between two such crossings there would have to be a
crossing in the decreasing direction. Hence consecutive zeros correspond to
consecutive values of $\pi/2+\pi\mathbb Z$. Matching the phase at infinity and
the numbering of zeros gives, for all $k\ge1$,
\[
P_d(\tau_{d,2k-1})-\delta_1(\tau_{d,2k-1})
=\frac{(2k-1)\pi}{2}.
\]
Since $\delta_1\le0$,
\[
P_d(\tau_{d,2k-1})\le\frac{(2k-1)\pi}{2} =P_d(\gamma_{d,2k-1}).
\]
The strict monotonicity of $P_d$ gives
\begin{equation}\label{eq:odd-zero-comparison}
\tau_{d,2k-1}\le\gamma_{d,2k-1}.
\end{equation}

For $u_{d,2}$, use the positively oriented basis
\[
w_1=y_{d}^{+},\quad w_2=-y_{d}^{-}.
\]
By~\eqref{eq:v-W}, $W(w_1,w_2)=4/\pi$, and if $Q_d=P_d-\pi/2$, then
\[
w_1=A_d\cos Q_d,\quad w_2=A_d\sin Q_d,
\quad Q_d'>0.
\]
Theorem~\ref{thm:true-tail} and~\eqref{eq:beta-tail} show that the phase shifts
of $u_{d,2}$ and $w_1$ at infinity agree modulo $\pi$. Therefore, after choosing
the sign of $u_{d,2}$, we have
\[
u_{d,2}(x)=c_2w_1(x)+o(1),
\quad
u_{d,2}'(x)=c_2w_1'(x)+o(1),
\quad c_2>0.
\]
Applying Lemma~\ref{lem:phase-comparison}, we get
\[
u_{d,2}=rA_d\cos(Q_d-\delta_2),
\quad \delta_2\le0.
\]
Taking into account the numbering of zeros,
\[
Q_d(\tau_{d,2k})-\delta_2(\tau_{d,2k}) =\frac{(2k-1)\pi}{2}.
\]
Hence
\[
P_d(\tau_{d,2k}) \le k\pi=P_d(\gamma_{d,2k}),
\]
and therefore
\begin{equation}\label{eq:even-zero-comparison}
\tau_{d,2k}\le\gamma_{d,2k}.
\end{equation}
Formulas~\eqref{eq:odd-zero-comparison} and \eqref{eq:even-zero-comparison}
prove the theorem.
\end{proof}

\subsection{Limiting distribution of zeros}
We prove Theorem~\ref{thm:limit-measure}. Recall that
\[
\mu_d=\frac1d\sum_{n\ge1}\delta_{\tau_{d,n}/d}.
\]
Let $M$ be the distribution function of the limiting measure
\eqref{eq:intro-measure}:
\begin{equation}\label{eq:M-def}
M(t)=\mu_*((0,t])=
\begin{cases}
0, & 0<t\le1,\\ (1/\pi)\bigl( \sqrt{t^2-1}-\arccos(1/t) \bigr), & t>1.
\end{cases}
\end{equation}

\begin{lem}[{\cite[Lemma~8]{Go26a}}]\label{lem:zero-count}
Let a nonzero solution $u$ of the equation $u''+V(x)u=0$ be defined on an
interval of length $L$, and suppose $V(x)\le M^2$ on this interval. Then the
number of zeros of $u$ does not exceed $ML/\pi+1$.
\end{lem}

\begin{proof}[Proof of Theorem~\ref{thm:limit-measure}]
Denote
\[
N_d(t)=\bigl|\{n\colon \tau_{d,n}\le dt\}\bigr|,
\quad
\widetilde N_d(t)=\bigl|\{n\ge1\colon \gamma_{d,n}\le dt\}\bigr|.
\]
From~\eqref{eq:beta-def} and the Debye asymptotics of the Bessel phase, for each
fixed $t>1$,
\[
\frac{\widetilde N_d(t)}d\to M(t),
\]
while for $0<t<1$ the left-hand side tends to zero. By
Theorem~\ref{thm:all-zero-comparison}, $N_d(t)\ge \widetilde N_d(t)$, and hence
\begin{equation}\label{eq:count-lower}
\liminf_{d\to\infty}\frac{N_d(t)}d\ge M(t).
\end{equation}

We prove the reverse estimate. If $t<1$, the potential in
\eqref{eq:Bessel-majorant} is negative on $(0,dt)$ for all sufficiently large
$d$. Each of the two solutions $u_{d,1},u_{d,2}$ has at most one zero there, so
$N_d(t)=O(1)$.

Let $t>1$ and $1=s_0<s_1<\cdots<s_N=t$. On
$[ds_{i-1},ds_i]$, by~\eqref{eq:Bessel-majorant},
\[
V_d(x)\le M_{d,i}^2,
\quad
M_{d,i}=\frac12\sqrt{\Bigl(1-\frac{1-d^{-2}}{s_i^2}\Bigr)_+}.
\]
Also, on $(0,d)$ each of the solutions $u_{d,1},u_{d,2}$ has at most one zero.
Indeed, by~\eqref{eq:Bessel-majorant}, for $0<x\le d$,
\[
V_d(x)\le V_{d,0}(x)\le \frac1{4d^2}.
\]
Applying Lemma~\ref{lem:zero-count} on $[\varepsilon,d]$ with
$M=1/(2d)$ gives
\[
N_{[\varepsilon,d]}(u_{d,j})
\le
\frac{d-\varepsilon}{2\pi d}+1<2,
\quad j=1,2.
\]
Thus the number of zeros in $(0,d]$ not included in the partition of $[d,dt]$
is $O(1)$.

Applying Lemma~\ref{lem:zero-count} to both solutions, we obtain
\[
\limsup_{d\to\infty}\frac{N_d(t)}d
\le\frac2\pi\sum_{i=1}^{N}(s_i-s_{i-1})
\lim_{d\to\infty}M_{d,i}.
\]
Refining the partition,
\begin{equation}\label{eq:count-upper}
\limsup_{d\to\infty}\frac{N_d(t)}d
\le\frac1\pi\int_1^t\sqrt{1-s^{-2}}\,ds=M(t).
\end{equation}
Formulas~\eqref{eq:count-lower} and \eqref{eq:count-upper} give
\[
\frac{N_d(t)}d\to M(t)
\]
for all $t>0$. Since
\[
\frac{N_d(t)}d=\mu_d((0,t]),
\quad
M(t)=\mu_*((0,t]),
\]
this is exactly the weak convergence on compact subsets
$\mu_d\to\mu_*$.
\end{proof}

\subsection{Root asymptotics}

We show that the limiting measure alone already determines the exponential
scale of the constant. Put
\[
\alpha_d=\frac1d\,\log\frac{a_d}{d^d}.
\]
By~\eqref{eq:L-via-a} and Stirling's formula,
\begin{equation}\label{eq:L-alpha}
\frac1d\,\log\Ld(d)=-1-\alpha_d+o(1).
\end{equation}

We need two simple uniform estimates. Since $\varphi_d$ has exponential
type~$1$ and $|\varphi_d(x)|\le1$ on $\mathbb R$; see
\cite[Sec.~2.3]{Go26a}, the standard estimate for functions of Bernstein class
gives
\[
|\varphi_d(x+iy)|\le e^{|y|}.
\]
Applying Jensen's formula to the disk of radius $2R$ and using
$\varphi_d(0)=1$, we obtain
\[
2\log2\,|\{n\colon \tau_{d,n}\le R\}|
\le
\frac1{2\pi}\int_0^{2\pi}
\log|\varphi_d(2Re^{i\theta})|\,d\theta
\le
\frac{4R}{\pi}.
\]
Hence
\[
|\{n\colon \tau_{d,n}\le R\}|\le\frac{2R}{\pi\log2}.
\]
Therefore
\begin{equation}\label{eq:tail-bound}
\sup_d\int_T^\infty\frac{d\mu_d(t)}{t^2}\le\frac CT.
\end{equation}
Also, the zero equilibrium condition \eqref{eq:zero-equilibrium} with $n=1$
gives $a_d<\tau_{d,1}^d$. From \eqref{eq:Lstar} and the definition
\eqref{eq:ad},
\[
a_d=\frac{(2\pi)^d}{2v_d^2\Ld(d)}.
\]
Thus the upper estimate in~\eqref{eq:known-bounds} and Stirling's formula
\[
v_d^{2/d}=\frac{2\pi e}{d}\,(1+o(1))
\]
give
\begin{equation}\label{eq:first-zero-apriori}
\liminf_{d\to\infty}\frac{\tau_{d,1}}d
\ge\frac1{\sqrt{2e}}>0.
\end{equation}

For an interval $J=[p,q]\subset(0,1)$, introduce
\[
\mathcal K_J(t)=\frac1{|J|}\int_J
\log\Bigl|1-\frac{s^2}{t^2}\Bigr|\,ds.
\]
By~\eqref{eq:first-zero-apriori}, \eqref{eq:tail-bound}, and the weak
convergence in Theorem~\ref{thm:limit-measure}, we have
\begin{equation}\label{eq:averaged-potential}
\frac1{|J|}\int_J\frac1d\,\log|\varphi_d(ds)|ds
\to\int_1^\infty\mathcal K_J(t)d\mu_*(t).
\end{equation}
Here the logarithmic singularity is removed by integration in~$s$, and the tail
is uniformly small by~\eqref{eq:tail-bound}.

Put
\[
F_*(s)=\int_1^\infty
\log\Bigl(1-\frac{s^2}{t^2}\Bigr)d\mu_*(t),\quad 0<s<1.
\]
Differentiation under the integral sign gives
\[
F_*(s)=\sqrt{1-s^2}-1-
\log\frac{1+\sqrt{1-s^2}}2.
\]
Consequently,
\begin{equation}\label{eq:g-star}
g_*(s)=\log s+F_*(s)\nearrow\log2-1
\quad(s\uparrow1).
\end{equation}

Radial integration and the definition of $a_d$ give the exact identity
\[
2a_d=d\int_0^\infty|\varphi_d(x)|x^{d-1}dx.
\]
After the substitution $x=ds$, restriction of the integral to $J$, and
application of Jensen's inequality, we obtain
\[
\alpha_d\ge\frac1{|J|}\int_J\Bigl(
\frac1d\,\log|\varphi_d(ds)|+
\Bigl(1-\frac1d\Bigr)\log s\Bigr)ds+o(1).
\]
Formulas~\eqref{eq:averaged-potential} and \eqref{eq:g-star} give, for every
fixed interval $J\subset(0,1)$,
\[
\liminf_{d\to\infty}\alpha_d
\ge
\frac1{|J|}\int_J g_*(s)\,ds.
\]
Taking $J=[1-2\varepsilon,1-\varepsilon]$, $0<\varepsilon<1/2$, and then
letting $\varepsilon\downarrow0$, we obtain
\[
\liminf_{d\to\infty}\alpha_d
\ge
\lim_{\varepsilon\downarrow0}
\frac1\varepsilon
\int_{1-2\varepsilon}^{1-\varepsilon}g_*(s)\,ds
=
g_*(1-)
=
\log2-1.
\]

On the other hand, the known estimate $\Ld(d)\ge2^{-d}$ and
\eqref{eq:L-alpha} give
\[
\limsup_{d\to\infty}\alpha_d\le\log2-1.
\]
Hence
\[
\alpha_d\to\log2-1.
\]
By~\eqref{eq:L-alpha}, this immediately implies
\[
\Ld(d)^{1/d}\to\frac12.
\]
Thus Corollary~\ref{cor:intro-root-asymptotic} is proved.

\subsection{Relative canonical products}
Put
\[
\lambda_{d,n}=\pi\Bigl(n+\frac d2\Bigr).
\]
By Theorem~\ref{thm:true-tail}, the product
\[
D_d=\prod_{n=1}^{\infty}
\frac{\lambda_{d,n}^2}{\tau_{d,n}^2}
\]
converges absolutely. Similarly, by~\eqref{eq:beta-tail},
\begin{equation}\label{eq:DB}
\widetilde D_d=\prod_{n=1}^{\infty}
\frac{\lambda_{d,n}^2}{\gamma_{d,n}^2}
\end{equation}
converges. Theorem~\ref{thm:all-zero-comparison} immediately gives
\begin{equation}\label{eq:Dcompare}
D_d\ge \widetilde D_d.
\end{equation}

The product $D_d$ can be evaluated exactly.

\begin{lem}\label{lem:relative-derivative}
Let $c>-1$,
\[
\rho_n=\pi(n+c),\quad n\ge1,
\]
and let a strictly increasing sequence $x_n>0$ satisfy
\begin{equation}\label{eq:relative-assumption}
x_n-\rho_n=O(n^{-1}).
\end{equation}
Put
\[
P(z)=\prod_{n=1}^{\infty} \Bigl(1-\frac{z^2}{x_n^2}\Bigr), \quad
Q(z)=\prod_{n=1}^{\infty} \Bigl(1-\frac{z^2}{\rho_n^2}\Bigr)
\]
and
\[
D=\prod_{n=1}^{\infty}\frac{\rho_n^2}{x_n^2}.
\]
Then $D$ converges absolutely and
\begin{equation}\label{eq:relative-derivative-limit}
\frac{|P'(x_n)|}{|Q'(\rho_n)|}\to D, \quad n\to\infty.
\end{equation}
\end{lem}

\begin{proof}
From~\eqref{eq:relative-assumption} and $\rho_n\asymp n$, we have
\[
\frac{x_n}{\rho_n}=1+O(n^{-2}).
\]
Therefore,
\[
\log\frac{\rho_n^2}{x_n^2}=O(n^{-2}),
\]
and the product $D$ converges absolutely. The canonical products $P$ and $Q$
also converge locally uniformly because
\[
\sum_{n=1}^{\infty}\frac1{x_n^2}<\infty, \quad
\sum_{n=1}^{\infty}\frac1{\rho_n^2}<\infty.
\]

Differentiating the products at their corresponding zeros, we obtain the exact
identity
\begin{equation}\label{eq:derivative-ratio-product}
\frac{|P'(x_n)|}{|Q'(\rho_n)|}
=
D\,\frac{x_n}{\rho_n}
\prod_{k\ne n}
\Bigl|
\frac{x_k^2-x_n^2}{\rho_k^2-\rho_n^2}
\Bigr|.
\end{equation}
Put $e_k=x_k^2-\rho_k^2$. By~\eqref{eq:relative-assumption},
$e_k=O(1)$. For $k\ne n$,
\[
\frac{x_k^2-x_n^2}{\rho_k^2-\rho_n^2}
=
1+\frac{e_k-e_n}{\rho_k^2-\rho_n^2}.
\]
Since both sequences are strictly increasing, the left-hand side is positive.
Also,
\[
|\rho_k^2-\rho_n^2|
=
\pi^2|k-n|(k+n+2c),
\]
and therefore
\[
\sum_{k\ne n}
\frac{|e_k-e_n|}{|\rho_k^2-\rho_n^2|}
\le
C\sum_{k\ne n}
\frac1{|k-n|(k+n+2c)}
=
O_c\Bigl(\frac{\log n}{n}\Bigr).
\]
In particular,
\[
\max_{k\ne n}
\frac{|e_k-e_n|}{|\rho_k^2-\rho_n^2|}
\to0.
\]
Consequently,
\[
\sum_{k\ne n}
\Bigl|
\log\Bigl(
1+\frac{e_k-e_n}{\rho_k^2-\rho_n^2}
\Bigr)
\Bigr|
=
O_c\Bigl(\frac{\log n}{n}\Bigr)
=o(1),
\]
and hence
\[
\prod_{k\ne n}
\Bigl|
\frac{x_k^2-x_n^2}{\rho_k^2-\rho_n^2}
\Bigr|
=1+o(1).
\]
Finally, $x_n/\rho_n\to1$, and \eqref{eq:derivative-ratio-product} gives
\eqref{eq:relative-derivative-limit}.
\end{proof}

Consider the function
\[
F_d(z)=
\frac{\Gamma(d/2+1)^2}
{\Gamma(d/2+1+z/\pi)\Gamma(d/2+1-z/\pi)}.
\]
It is normalized by $F_d(0)=1$ and has zeros $\pm\lambda_{d,n}$. Hence
\[
F_d(z)=\prod_{n=1}^{\infty}
\Bigl(1-\frac{z^2}{\lambda_{d,n}^2}\Bigr).
\]
Using the pole of the gamma function at $1-n$, we obtain
\begin{equation}\label{eq:Fprime-lambda}
|F_d'(\lambda_{d,n})|
=\frac{\Gamma(d/2+1)^2\Gamma(n)}
{\pi\Gamma(n+d+1)}.
\end{equation}

\begin{lem}
For every $d\ge1$,
\begin{equation}\label{eq:D-normalization}
D_d= \frac{2a_d}{\pi^d\Gamma(d/2+1)^2}.
\end{equation}
Consequently,
\begin{equation}\label{eq:L-via-D}
2^d\Ld(d)=
\frac{(4/\pi)^d}{D_d}.
\end{equation}
\end{lem}

\begin{proof}
Apply Lemma~\ref{lem:relative-derivative} to
$x_n=\tau_{d,n}$, $\rho_n=\lambda_{d,n}$, using
Theorem~\ref{thm:true-tail}. Then
\[
D_d=\lim_{n\to\infty}
\frac{|\varphi_d'(\tau_{d,n})|}{|F_d'(\lambda_{d,n})|}.
\]
By~\eqref{eq:residue} and \eqref{eq:Fprime-lambda},
\begin{align*}
\frac{|\varphi_d'(\tau_{d,n})|}{|F_d'(\lambda_{d,n})|}
&=%
\frac{2a_d}{\tau_{d,n}^{d+1}}\,
\frac{\pi\Gamma(n+d+1)}
{\Gamma(d/2+1)^2\Gamma(n)}.
\end{align*}
Since
\[
\tau_{d,n}=\pi\Bigl(n+\frac d2\Bigr)(1+O_d(n^{-2}))
\]
and
\[
\frac{\Gamma(n+d+1)}{\Gamma(n)}=n^{d+1}(1+o(1)),
\]
the limit is the right-hand side of \eqref{eq:D-normalization}.
Formula~\eqref{eq:L-via-D} follows directly from
\eqref{eq:D-normalization} and~\eqref{eq:L-via-a}.
\end{proof}

From~\eqref{eq:Dcompare} and~\eqref{eq:L-via-D}, we obtain the explicit upper
estimate
\[
2^d\Ld(d)
\le
\frac{(4/\pi)^d}{\widetilde D_d}.
\]
It remains to find the asymptotics of $\widetilde D_d$.

\subsection{Model product}
Return to the function $M$ defined in~\eqref{eq:M-def}. We have
\[
M'(t)=\frac1\pi\,\sqrt{1-t^{-2}}, \quad t>1.
\]
Thus $M$ is strictly increasing on $[1,\infty)$ from $0$ to $+\infty$.
Denote its inverse by
\[
T=M^{-1}\colon [0,\infty)\to[1,\infty).
\]
Define the quantile grid
\[
\rho_{d,n}=d\,T(n/d),\quad n\ge1,
\]
so that
\begin{equation}\label{eq:rho-quantization}
dM(\rho_{d,n}/d)=n.
\end{equation}

Introduce the model product
\begin{equation}\label{eq:D0}
D_d^0=\prod_{n=1}^{\infty}
\frac{\lambda_{d,n}^2}{\rho_{d,n}^2}.
\end{equation}

\begin{lem}\label{lem:D0}
As $d\to\infty$,
\begin{equation}\label{eq:D0-asymp}
D_d^0=
\Bigl(\frac2\pi+o(1)\Bigr)
\Bigl(\frac4\pi\Bigr)^d.
\end{equation}
\end{lem}

For the proof, we need asymptotics of the inverse function $T$, an exact
evaluation of one integral, and the Euler--Maclaurin formula.

\subsubsection*{Asymptotics of the inverse function}
From~\eqref{eq:M-def}, as $s\downarrow0$,
\[
M(1+s)=\frac{2\sqrt2}{3\pi}\,s^{3/2} +O(s^{5/2}).
\]
Consequently,
\begin{equation}\label{eq:T-edge}
T(u)=1+
\Bigl(\frac{3\pi}{2\sqrt2}\Bigr)^{2/3}u^{2/3}
+O(u^{4/3}),\quad u\downarrow0.
\end{equation}
As $t\to\infty$,
\[
\pi\Bigl(M(t)+\frac12\Bigr)
=t+\frac1{2t}+O(t^{-3}),
\]
and inversion gives
\begin{equation}\label{eq:T-infty}
T(u)=\pi\Bigl(u+\frac12\Bigr)
-\frac{1}{2\pi(u+1/2)}+O(u^{-3}).
\end{equation}

Put
\begin{equation}\label{eq:f-def}
f(u)=2\log\frac{\pi(u+1/2)}{T(u)},\quad u\ge0.
\end{equation}
From~\eqref{eq:T-edge} and \eqref{eq:T-infty}, we obtain
\begin{equation}\label{eq:f0}
f(0)=2\log\frac\pi2,
\end{equation}
and
\[
f(u)=
\begin{cases}
f(0)+O(u^{2/3}), & u\downarrow0,\\
O(u^{-2}), & u\to\infty.
\end{cases}
\]
Moreover, from the identity $M(T(u))=u$,
\[
T'(u)=\frac{\pi}{\sqrt{1-T(u)^{-2}}},
\]
so
\[
f'(u)=
\begin{cases}
O(u^{-1/3}), & u\downarrow0,\\
O(u^{-3}), & u\to\infty.
\end{cases}
\]
In particular, $f'\in L^1(0,\infty)$.

\subsubsection*{The integral of $f$}
We show that for the function~\eqref{eq:f-def},
\begin{equation}\label{eq:f-integral}
\int_0^\infty f(u)\,du=\log\frac4\pi.
\end{equation}
Indeed, put $t=T(u)=\csc y$, $0<y\le\pi/2$. Then
\[
\sqrt{t^2-1}=\cot y,
\quad
\arccos(1/t)=\frac\pi2-y.
\]
By~\eqref{eq:M-def},
\[
\pi\Bigl(u+\frac12\Bigr)=\cot y+y,
\quad
du=-\frac{\cot^2y}{\pi}\,dy.
\]
Also,
\[
\frac{\pi(u+1/2)}{T(u)}
=\cos y+y\sin y=F(y).
\]
Therefore,
\begin{equation}\label{eq:I-y}
\int_0^\infty f(u)du
=\frac2\pi\int_0^{\pi/2}
\cot^2y\,\log F(y)\,dy.
\end{equation}
We have $F'(y)=y\cos y$ and
$\int\cot^2y\,dy=-\cot y-y$. Integration by parts gives
\begin{align*}
\int_0^{\pi/2}\cot^2y\log F(y)\,dy
&=-\frac\pi2\,\log\frac\pi2
+\int_0^{\pi/2}y\cot y\,dy.
\end{align*}
The boundary term at $y=0$ is zero because
$\log F(y)=O(y^2)$. Finally,
\[
\int_0^{\pi/2}y\cot y\,dy
=-\int_0^{\pi/2}\log\sin y\,dy
=\frac\pi2\,\log2.
\]
Substitution into~\eqref{eq:I-y} gives~\eqref{eq:f-integral}.

\subsubsection*{Euler--Maclaurin formula}
We need the following simple form of the Euler--Maclaurin formula.

\begin{lem}\label{lem:Euler}
Let $f$ be absolutely continuous on compact intervals, let
$f,f'\in L^1(0,\infty)$, and suppose that $f(0+)$ exists. Then, as
$h\downarrow0$,
\[
\sum_{n=1}^{\infty}f(nh)
=
\frac1h\int_0^\infty f(u)\,du-\frac{f(0+)}2+o(1).
\]
\end{lem}

\begin{proof}
The first Euler--Maclaurin formula gives
\[
\sum_{n=1}^{\infty}f(nh)
=
\frac1h\int_0^\infty f(u)\,du-\frac{f(0+)}2
+\int_0^\infty B_1(\{u/h\})f'(u)\,du,
\]
where $B_1(s)=s-1/2$ on $(0,1)$ and is extended periodically. The last
integral tends to zero. This is first immediate for step $L^1$ functions in
place of $f'$, since $B_1$ has mean zero, and then follows for $f'$ by density
and boundedness of~$B_1$.
\end{proof}

\begin{proof}[Proof of Lemma~\ref{lem:D0}]
From the definitions~\eqref{eq:D0} and \eqref{eq:f-def},
\[
\log D_d^0=\sum_{n=1}^{\infty}f(n/d).
\]
Apply Lemma~\ref{lem:Euler} with $h=1/d$, and then use
\eqref{eq:f-integral} and \eqref{eq:f0}:
\[
\log D_d^0
=d\log\frac4\pi-\log\frac\pi2+o(1)
=d\log\frac4\pi+\log\frac2\pi+o(1).
\]
Exponentiating, we obtain~\eqref{eq:D0-asymp}.
\end{proof}

\subsection{Uniform Bessel phase}
It remains to compare the exact Bessel nodes $\gamma_{d,n}$ with the quantile
nodes $\rho_{d,n}$. For $t\ge1$, put
\begin{equation}\label{eq:action-A}
A(t)=
\sqrt{t^2-1}-\arccos(1/t)
=
\pi M(t).
\end{equation}
For $\nu=d/2$, define the phase deviation
\begin{equation}\label{eq:rnu}
r_\nu(t)=
\theta_\nu(\nu t)+\frac\pi4-\nu A(t),
\quad t>1.
\end{equation}

\begin{lem}
There exists a constant $C>0$ such that for $\nu\ge1/2$:

\begin{enumerate}[\rm (i)]
\item for all $t>1$,
\begin{equation}\label{eq:r-global}
-\frac Ct\le r_\nu(t)<0;
\end{equation}

\item for any $1<t_0<t_1<\infty$,
\begin{equation}\label{eq:r-bulk}
\sup_{t_0\le t\le t_1}|r_\nu(t)|\to0,
\quad \nu\to\infty.
\end{equation}
\end{enumerate}
\end{lem}

\begin{proof}
Put
\[
\widetilde\theta_\nu(x)=
\sqrt{x^2-\nu^2}
-\nu\arccos\frac{\nu}{x}
-\frac\pi4.
\]
The uniform enclosure for the Bessel phase \cite[Theorem~1.4]{Fi24} gives, for
$x>\nu\ge0$,
\[
\max\,\Bigl\{
\widetilde\theta_\nu(x)
-\frac{3x^2+2\nu^2}{24(x^2-\nu^2)^{3/2}},
-\frac\pi2
\Bigr\}
<
\theta_\nu(x)
<
\widetilde\theta_\nu(x).
\]
Substituting $x=\nu t$ and using~\eqref{eq:action-A}, we get
\begin{equation}\label{eq:r-enclosure}
0<-r_\nu(t)<
\min\,\Bigl\{
\frac{3t^2+2}{24\nu(t^2-1)^{3/2}},
\ \nu A(t)+\frac\pi4
\Bigr\}.
\end{equation}

To obtain~\eqref{eq:r-global}, put $q=\sqrt{t^2-1}$. Then
$A(t)=q-\arctan q$, and the function
\[
\frac{(3q^2+5)A(t)}{q^3}
\]
is bounded for $q>0$. Therefore, with $y=\nu A(t)$, the first term in the
minimum in~\eqref{eq:r-enclosure} is bounded by $C/y$, and
\[
0<-r_\nu(t)
\le
\min\Bigl\{y+\frac\pi4,\frac Cy\Bigr\}
\le C.
\]
For $1<t<2$, after increasing $C$, this gives
\[
0<-r_\nu(t)\le\frac Ct.
\]
For $t\ge2$, the first quantity under the minimum in
\eqref{eq:r-enclosure} does not exceed $C/(\nu t)\le C/t$. Together with the
estimate for $1<t<2$, this proves~\eqref{eq:r-global}.

On each interval $[t_{0},t_{1}]$, the first term in
\eqref{eq:r-enclosure} is $O_{t_0,t_1}(\nu^{-1})$, so
\[
\sup_{t_{0}\le t\le t_{1}}|r_\nu(t)|\to0,
\]
which gives~\eqref{eq:r-bulk}.
\end{proof}

Define
\begin{equation}\label{eq:epsilon}
\varepsilon_{d,n}
=dM\Bigl(\frac{\gamma_{d,n}}d\Bigr)-n.
\end{equation}
Since $\nu=d/2$ and, by~\eqref{eq:beta-def},
$P_d(\gamma_{d,n})=\pi n/2$, formula~\eqref{eq:rnu} gives the exact equality
\begin{equation}\label{eq:epsilon-r}
\varepsilon_{d,n}
=-\frac2\pi\,r_{d/2}\Bigl(\frac{\gamma_{d,n}}d\Bigr).
\end{equation}
In particular, by~\eqref{eq:r-global},
\begin{equation}\label{eq:epsilon-uniform}
0\le\varepsilon_{d,n}\le C
\end{equation}
uniformly for $d,n\ge1$.

\begin{lem}\label{lem:epsilon-l2}
As $d\to\infty$,
\[
\sum_{n=1}^{\infty}\varepsilon_{d,n}^2=o(d).
\]
\end{lem}

\begin{proof}
Fix $0<\delta<1<A$. If $\delta d\le n\le Ad$, then by
\eqref{eq:epsilon-uniform} and the definition of $\varepsilon_{d,n}$, the
quantities $M(\gamma_{d,n}/d)$ remain in a fixed compact subset of
$(0,\infty)$; hence $\gamma_{d,n}/d$ remain in a compact subset of
$(1,\infty)$. By~\eqref{eq:r-bulk},
\[
\max_{\delta d\le n\le Ad}\varepsilon_{d,n}=o(1).
\]

For $n\ge Ad$ and sufficiently large $A$, from
\[
n=dM(t)+O(1),
\quad
t=\frac{\gamma_{d,n}}d,
\]
and the asymptotics $M(t)\asymp t$ for $t\ge2$, it follows that
$t\asymp n/d$. Hence, by~\eqref{eq:r-global} and \eqref{eq:epsilon-r},
\[
0\le\varepsilon_{d,n}\le C\,\frac d n,
\quad n\ge Ad.
\]
Now
\[
\frac1d\sum_{n\ge1}\varepsilon_{d,n}^2
\le C\delta
+\frac1d\sum_{\delta d\le n\le Ad}o(1)
+Cd\sum_{n\ge Ad}\frac1{n^2}
\le C\delta+o(1)+\frac CA.
\]
First let $d\to\infty$, then $\delta\downarrow0$ and $A\to\infty$.
\end{proof}

\subsection{Asymptotics of the product $\widetilde D_{d}$}
From~\eqref{eq:rho-quantization} and \eqref{eq:epsilon},
\[
\rho_{d,n}=dT(n/d),
\quad
\gamma_{d,n}=dT\Bigl(\frac{n+\varepsilon_{d,n}}d\Bigr).
\]

For $\xi>0$, put
\[
L_d(\xi)=\log\bigl(dT(\xi/d)\bigr).
\]
Differentiating $M(T(u))=u$ and using
$M'(t)=\pi^{-1}\sqrt{1-t^{-2}}$, we obtain the exact formula
\[
L_d'(\xi)
=
\frac{\pi}{d\sqrt{T(\xi/d)^2-1}}.
\]

Put
\[
w_{d,n}=
\begin{cases}
d^{-2/3}n^{-1/3}, & 1\le n\le d,\\
n^{-1}, & n>d.
\end{cases}
\]
Then
\begin{equation}\label{eq:weight-bound}
0\le
\log\frac{\gamma_{d,n}}{\rho_{d,n}}
\le
C\varepsilon_{d,n}w_{d,n}.
\end{equation}

Indeed, since $\varepsilon_{d,n}\ge0$, the mean value theorem gives
\[
\log\frac{\gamma_{d,n}}{\rho_{d,n}}
=
\varepsilon_{d,n}L_d'(\xi_{d,n}),
\quad
n\le\xi_{d,n}\le n+\varepsilon_{d,n}.
\]
By~\eqref{eq:T-edge} and continuity of $T$,
\[
T(u)^2-1\ge c u^{2/3},
\quad 0<u\le2.
\]
Therefore, for $1\le n\le d$ and all sufficiently large $d$,
\[
L_d'(\xi_{d,n})
\le
C d^{-2/3}\xi_{d,n}^{-1/3}
\le
C d^{-2/3}n^{-1/3}.
\]
For $n>d$, the function
\[
u\mapsto \frac{\pi u}{\sqrt{T(u)^2-1}}
\]
is bounded on $[1,\infty)$, and hence
\[
L_d'(\xi_{d,n})
\le
\frac C{\xi_{d,n}}
\le
\frac Cn.
\]
This proves~\eqref{eq:weight-bound}.

Moreover,
\begin{equation}\label{eq:weights-l2}
\sum_{n=1}^{\infty}w_{d,n}^2
\le
d^{-4/3}\sum_{n\le d}n^{-2/3}
+
\sum_{n>d}n^{-2}
=
O(d^{-1}).
\end{equation}

\begin{lem}\label{lem:DB-asymp}
As $d\to\infty$,
\[
\widetilde D_d=
\Bigl(\frac2\pi+o(1)\Bigr)
\Bigl(\frac4\pi\Bigr)^d.
\]
\end{lem}

\begin{proof}
From~\eqref{eq:DB} and \eqref{eq:D0},
\[
-\log\frac{\widetilde D_d}{D_d^0}
=
2\sum_{n=1}^{\infty}
\log\frac{\gamma_{d,n}}{\rho_{d,n}}.
\]
By~\eqref{eq:weight-bound}, the Cauchy--Schwarz inequality,
Lemma~\ref{lem:epsilon-l2}, and~\eqref{eq:weights-l2},
\[
0\le
-\log\frac{\widetilde D_d}{D_d^0}
\le
C
\biggl(\,\sum_{n\ge1}\varepsilon_{d,n}^2\biggr)^{1/2}
\biggl(\,\sum_{n\ge1}w_{d,n}^2\biggr)^{1/2}
=o(1).
\]
Hence
\[
\widetilde D_d=D_d^0(1+o(1)),
\]
and the statement follows from Lemma~\ref{lem:D0}.
\end{proof}

\subsection{Completion of the upper estimate}
From~\eqref{eq:Dcompare}, \eqref{eq:L-via-D}, and
Lemma~\ref{lem:DB-asymp}, we obtain
\[
2^d\Ld(d)
=
\frac{(4/\pi)^d}{D_d}
\le
\frac{(4/\pi)^d}{\widetilde D_d}
=
\frac\pi2+o(1).
\]
Thus
\[
\limsup_{d\to\infty}2^d\Ld(d)\le\frac\pi2.
\]

Together with the lower estimate~\eqref{eq:lower-second}, this gives
\[
2^d\Ld(d)\to\frac\pi2,
\]
which completes the proof of Theorem~\ref{thm:main}.

\subsection*{Acknowledgments}
The author is grateful to the artificial intelligence model used in preparing
this paper for its assistance.


\begin{thebibliography}{99}

\bibitem{Da21}
F.~Dai, D.~Gorbachev, S.~Tikhonov,
\textit{Estimates of the asymptotic Nikolskii constants for spherical
polynomials}, J.~Complexity \textbf{65} (2021), 101553.

\bibitem{Fi24}
N.~Filonov, M.~Levitin, I.~Polterovich, D.~A.~Sher, \textit{Uniform enclosures
for the phase and zeros of Bessel functions and their derivatives}, SIAM J.
Math. Anal. \textbf{56} (2024), no.~6, 7644--7682.

\bibitem{Gon26}
F.~Gon\c{c}alves, D.~Radchenko, A.~Pedro Ramos, \textit{The
H\"ormander--Bernhardsson function in higher dimensions}, arXiv:2608.22198,
2026.

\bibitem{Go26a}
D.~V.~Gorbachev, \textit{The Nikolskii constant in odd dimensions},
arXiv:2608.15674, 2026.

\bibitem{Go26}
D.~V.~Gorbachev, \textit{The Nikolskii Constant in Arbitrary Dimension},
arXiv:2608.22578, 2026.

\bibitem{Gr08}
L.~Grafakos, \textit{Classical Fourier Analysis}, 2nd ed. Graduate Texts in
Mathematics, vol.~249. Springer, New York, 2008.

\bibitem{Ha02}
P.~Hartman, \emph{Ordinary Differential Equations}, 2nd ed., Classics in
Applied Mathematics, vol.~38, SIAM, Philadelphia, PA, 2002.

\bibitem{DLMF}
NIST Digital Library of Mathematical Functions, https://dlmf.nist.gov

\bibitem{Ol97}
F.~W.~J.~Olver, \textit{Asymptotics and Special Functions}, A K Peters,
Wellesley, MA, 1997.

\bibitem{RS02}
Q.~I.~Rahman and G.~Schmeisser, \textit{Analytic Theory of Polynomials}, London
Mathematical Society Monographs, New Series, vol.~26, Clarendon Press, Oxford,
2002.

\bibitem{Wa66}
G.~N.~Watson, \textit{A Treatise on the Theory of Bessel Functions}, 2nd ed.,
Cambridge University Press, Cambridge, 1966.

\end{thebibliography}
\end{document}